\documentclass[12pt]{article}

\usepackage{amsmath, amssymb, amsthm, amsfonts, mathtools, array, xfrac, float, multirow, indentfirst, url, enumerate, comment, fullpage, afterpage, pifont, makecell, tcolorbox}

\usepackage[numbers]{natbib}

\providecommand{\noopsort}[1]{}
\usepackage{hyperref}
\hypersetup{hypertexnames=false} 
\usepackage{booktabs, siunitx, xcolor, graphicx, titlesec}

\newtheorem{thm}{Theorem}

\newtheorem{Lem}{Lemma}
\newtheorem{cor}{Corollary}
\newtheorem*{cor*}{Corollary}

\newtheorem{manualCorN}{Corollary}
\newenvironment{manualCor}[1]{%
  \renewcommand\themanualCorN{#1}%
  \manualCorN
}{\endmanualCorN}

\newtheorem{manualThmN}{Theorem}
\newenvironment{manualThm}[1]{%
  \renewcommand\themanualThmN{#1}%
  \manualThmN
}{\endmanualThmN}

\title{An explicit mean-value estimate for the PNT in intervals}
\author{\textsc{Michaela Cully-Hugill\footnote{m.cully-hugill@adfa.edu.au}} \\ School of Science \\ UNSW Canberra \\ Australia ACT 2612 \\   \ \\\textsc{Adrian W. Dudek\footnote{a.dudek@uq.edu.au}} \\ School of Mathematics and Physics \\ University of Queensland \\ Australia QLD 4072}

\begin{document}

\maketitle
\renewcommand{\bibname}{References}

\begin{abstract}
    This paper gives an explicit version of Selberg's mean-value estimate for the prime number theorem in intervals, assuming the Riemann hypothesis \cite{Selberg_43}. Two applications are given to short-interval results for primes and for Goldbach numbers. Under the Riemann hypothesis, we show there exists a prime in $(y,y+32277\log^2 y]$ for at least half the $y\in[x,2x]$ for all $x\geq 2$, and at least one Goldbach number in $(x,x+9696 \log^2 x]$ for all $x\geq 2$.
\end{abstract}

\section{Introduction}

Selberg's 1943 paper \cite{Selberg_43} features conditional and unconditional estimates for the asymptotic behaviour of the prime number theorem in short intervals $(x,x+h]$ with $h=o(x)$. They are reached via the relationship between Chebyshev's prime counting functions, $\theta(x)$ and $\psi(x)$, and the Riemann zeta-function, $\zeta(s)$. A notable waypoint in Selberg's method is an estimate for 
\begin{align*}
    J(x,\delta) = \int_1^{x} |\theta(y+\delta y)- \theta(y) - \delta y|^2 dy
\end{align*}
for $\delta\in (0,1]$, to gauge the mean value of $\theta(x)$ in short intervals. There has been much interest in this integral since Selberg's paper, for its connection to the prime number theorem and prime gaps, and for estimates on the zeros of $\zeta(s)$ and Montgomery's pair-correlation function. 

Assuming the Riemann hypothesis (RH), the best estimate for $J(x,\delta)$ is
\begin{align}\label{Selberg-MV}
    J(x,\delta) \ll \delta x^2\log^2 x,
\end{align}
for all $\delta\in [1/x,1]$, from Selberg \cite{Selberg_43}\footnote{This estimate is given in the second display equation on page 172 of \cite{Selberg_43}.}. Saffari and Vaughan gave a similar result in Lemma 5 of \cite{S_V_77}, but used an averaging technique with the Riemann--von Mangoldt explicit formula. Unconditionally, one of the best results for $J(x,\delta)$ to date is from Zaccagnini \cite{Zacc_98} of $J(x,\delta) \ll x^3\delta^2$ for $\delta\in [x^{-5/6-\epsilon(x)},1]$ with $\epsilon(x)\rightarrow 0$ as $x\rightarrow \infty$. 

Selberg's result (\ref{Selberg-MV}) is actually deduced from an estimate for a similar integral over 1 to $\delta^{-1}$: see equation (13) in \cite{Selberg_43}. This integral has itself been separately studied, in part because it allows a better illustration of the relationship between the size of the interval and the asymptotic behaviour of the integral. For more details on this see, e.g., the introduction of \cite{C_C_C_M_22}. Another useful reference is Zaccagnini's review paper \cite{Zaccagnini_16}, which gives a survey of the literature surrounding $J(x,\delta)$. Also see Goldston et al. \cite{G_G_M_01} for a version of Selberg's proof of (\ref{Selberg-MV}) for $\psi(x)$.

The primary goal of this paper is to prove the following explicit version of (\ref{Selberg-MV}).
\begin{thm}\label{mean-value-thm}
    Assuming RH, for all $x\geq 10^8$ and any $\delta\in (0,10^{-8}]$ we have
    $$\int_1^x \left| \theta(y+\delta y)- \theta(y) - \delta y \right|^2 dy < 202 \delta x^2 \log^2x.$$
\end{thm}

Estimates for $J(x, \delta)$ can be used to comment on the measure of intervals that contain primes. It is usually said that an estimate holds for `almost all' $y\in[x,2x]$ if the exceptional set has measure $o(x)$ (see footnote 4 on page 161 of \cite{Selberg_43} for Selberg's definition). Under RH, Selberg's estimate implies that almost all intervals $[x,x+h]$ contain a prime for any positive increasing function $h=o(x)$ with $h/\log^2 x \rightarrow\infty$ (stated in \cite[Cor.~2]{Wolke_1983}). For comparison, the best unconditional result is from Jia \cite{Jia_96-primes}, of primes in almost all intervals of the form $[x,x+x^{\frac{1}{20}+\epsilon}]$, for any $\epsilon>0$. 

With an explicit estimate for $J(x,\delta)$, we can explicitly determine the measure of the number of intervals in some range which contain primes. Moreover, Theorem \ref{mean-value-thm} allows us to do so for any interval wider than $O(\log y)$. To demonstrate how this can be done, we prove the following corollary in Section \ref{proof-corollary}.
\begin{cor}\label{some-intervals}
    Assuming RH, the set of $y\in [x,2x]$ for which there is at least one prime in $(y, y+ 32277 \log^2 y]$ has a measure of at least $x/2$ for all $x\geq 2$.
\end{cor}
Similar statements to Corollary \ref{some-intervals} can be made with other short intervals. We chose this particular interval to make a comparison with a conjecture\footnote{Although the basis for this conjecture has been called into question, it is still considered likely that Cram{\'e}r's conjecture is true for powers of $\log p_n$ above 2. See \cite[pg.~23]{Granville_95} and \cite{Pintz_07}.} of Cram{\'e}r \cite{Cramer_36}, that the upper bound on gaps between consecutive primes, $p_{n+1}-p_n$, should be $O(\log^2 p_n)$. This would predict Corollary \ref{some-intervals} to be true for all $y\in[x,2x]$. Another comparison can be made with the result of Goldston, Pintz, and Y\i ld\i r\i m \cite[Thm.~1]{G_P_Y_11}, that for any fixed $\eta>0$ there is a positive proportion of $y\in[x,2x]$ for which $(y,y+\eta\log y]$ contains a prime as $x\rightarrow\infty$.



Selberg's result can also be used to deduce interval results for Goldbach numbers. A Goldbach number is defined as the sum of two odd primes. We have estimates for the number of Goldbach numbers in intervals, and for the smallest interval containing a Goldbach number. See Languasco \cite{Languasco_95_review} for a survey. Linnik \cite{Linnik_52} first used Hardy and Littlewood's circle method to prove, under RH, that there exist Goldbach numbers in $[x,x+H]$ with $H=O(\log^{3+\epsilon} x)$ for any $\epsilon>0$. Katai \cite{Katai_67} refined this to $H= O(\log^{2} x)$ using methods from \cite{Selberg_43}. Montgomery and Vaughan \cite[Thm.~2]{M_V_75} also proved this result, but used (\ref{Selberg-MV}). This result has also been proved in \cite{Goldston_90} and \cite[Cor.~1]{L_P_94} using other techniques. Going a step further, Goldston \cite{Goldston_90} showed that under RH and Montgomery's pair-correlation conjecture \cite{Montgomery_73} we can take $H=O(\log x)$, and under the same assumptions, Languasco \cite{Languasco_98} proved that there is a positive proportion of Goldbach numbers in this interval. For more recent work on the average number of Goldbach numbers in intervals, see \cite{G_S_2021_arxiv}.

We prove the following version of Montgomery and Vaughan's result in Section \ref{proof-goldbach-intervals}.
\begin{thm}\label{thm-Goldbach-intervals}
Assuming RH, there exists a Goldbach number in the interval $(x,x+ 9696 \log^2 x]$ for all $x\geq 2$.
\end{thm}


\section{An explicit version of Selberg's result}
\subsection{Preliminary lemmas}

In this section we prove a number of lemmas needed to prove Theorem \ref{mean-value-thm}. Here and hereafter, let $s=\sigma+it$ and let $\rho=\beta+i\gamma$ denote any non-trivial zero of $\zeta(s)$. Selberg's proof of (\ref{Selberg-MV}) requires a mean-value estimate for the logarithmic derivative of $\zeta(s)$ on RH. In particular, Lemma 4 of \cite{Selberg_43} states that for sufficiently large $T$ and $\sigma\in(1/2,3/4]$,
\begin{equation}\label{Selberg-zeta-MV}
\int_0^T \left| \frac{\zeta'}{\zeta}(\sigma+it) \right|^2 dt = O\left( \frac{T}{(\sigma-1/2)^2} \right).    
\end{equation}
There does not appear to be an explicit version of (\ref{Selberg-zeta-MV}), but Selberg's proof is effective. See also an estimate for a similar integral from Farmer \cite[Lem.~2]{Farmer_95}. There are explicit estimates for $\zeta'(s)/\zeta(s)$ in the critical strip, such as in \cite[Cor.~1(b)]{Simonic_arxiv} of the order $O((\log t)^{2(1-\sigma)} (\log\log t)^2)$, or Lemma 2.8 of \cite{Dudek_16p} of $O(\log^2t)$, but these would not give an estimate of the form (\ref{Selberg-zeta-MV}). We will make (\ref{Selberg-zeta-MV}) explicit in Lemma 4, by way of an explicit version of \cite[Lem.~3]{Selberg_43} in Lemma 3. The latter will need Lemmas 1 and 2.

\begin{Lem}{(Karatsuba, \& Korol\"{e}v \cite[Lem.~2]{K_K_2005})}\label{K-K-lemma}
For $|\sigma|\leq 2$ and $|t|\geq 10$, 
\begin{align}\label{zetas}
    \left| \frac{\zeta'}{\zeta}(s) - \sum_{\rho}\frac{1}{s-\rho} \right| \leq \frac{1}{2} \log |t| + 3.
\end{align}
\end{Lem}

\begin{Lem}{(Selberg \cite[Lem.~2]{Selberg_43})}\label{Lem2}
    For $x>1$, and $\Lambda(n)$ denoting the von Mangoldt function,
\begin{equation*}
    \Lambda_x(n)=
    \begin{cases}
        \Lambda(n),&1\leq n\leq x\\
        \Lambda(n)\frac{\log(x^2/n)}{\log x},&x\leq n\leq x^2. 
    \end{cases}
\end{equation*}
Then, for any $s\neq 1$ and not a zero of $\zeta(s)$,
\begin{align}\label{formula}
    -\frac{\zeta'}{\zeta}(s) &= \sum_{n<x^2}\frac{\Lambda_x(n)}{n^s} + \frac{x^{1-s}-x^{2(1-s)}}{(1-s)^2\log x} - \frac{1}{\log x}\sum_{q=1}^\infty\frac{x^{-2q-s}-x^{-2(2q+s)}}{(2q+s)^2} \\
    &\qquad - \frac{1}{\log x}\sum_{\rho}\frac{x^{\rho-s}-x^{2(\rho-s)}}{(s-\rho)^2}. \nonumber
\end{align}
\end{Lem}

\begin{Lem}\label{Lem3}
    Assume RH. For $x\geq x_0\geq 3$, $|t|\geq 10$, $\frac{1}{2} +\frac{\alpha}{\log x} \leq \sigma \leq 1$, and $\alpha\geq 0.722$,
    \begin{equation*}
        \left| \frac{\zeta'}{\zeta}(s) \right| \leq A_2 \left| \sum_{n<x^2} \frac{\Lambda_x(n)}{n^s} \right| + \frac{A_1 A_2 \left( \log |t| + 6 \right)}{2(\sigma-\frac{1}{2})\log x}  + \frac{A_1 A_2 x}{t^2 \log x} + \frac{A_2 c_1 e^{-\alpha}}{t^2 x^\frac{5}{2}\log x}.
    \end{equation*}
where $A_1=e^{-\alpha}+e^{-2\alpha}$, $A_2 = \frac{\alpha}{\alpha - A_1}$, and $c_1 = 5/4$ for all $x\geq 3$.
\end{Lem}
\begin{proof}
    We assume RH throughout, so $\rho=\frac{1}{2} +i\gamma$. Starting with Lemma \ref{Lem2}, let $Z_i$ for $i=1, 2, 3$ denote the last three terms on the RHS of (\ref{formula}). For interest, we follow a similar proof to that of Lemma 2 in \cite{Simonic_22}. Using (\ref{formula}) we will denote
    \begin{equation}\label{F(s,x)}
        F(s,x) = \left| \frac{\zeta'}{\zeta}(s) + \sum_{n<x^2} \frac{\Lambda_x(n)}{n^s} \right| \leq |Z_1| + |Z_2| + |Z_3| 
    \end{equation}
for $x>1$, $1/2<\sigma\leq 1$, and $|t|>0$. We have
    \begin{align*}
        |Z_1| &= \left| \frac{x^{1-s}-x^{2(1-s)}}{(1-s)^2\log x} \right| \leq \frac{x^{2-2\sigma} + x^{1-\sigma}}{t^2 \log x},
    \end{align*}
and, using the sum of a geometric series,
    \begin{align*}
        |Z_2| &= \left| \frac{1}{\log x} \sum_{q=1}^\infty \frac{x^{-(2q+s)}-x^{-2(2q+s)}}{(2q+s)^2} \right| \\
        &\leq \frac{1}{\log x}\left( \frac{x^{-(2+\sigma)} + x^{-2(2+\sigma)}}{t^2+4}  + \frac{1}{t^2+16} \sum_{q=2}^\infty \left( x^{-(2q+\sigma)} + x^{-2(2q+\sigma)} \right)  \right) \\
        &\leq \frac{c_1}{(t^2+4) x^{2+\sigma} \log x},
    \end{align*}
where $c_1 = 1 + 2/(x_0^2-1)$ for $x\geq x_0 \geq 3$. For $Z_3$ we will use Lemma \ref{K-K-lemma} and $$\text{Re}\left\lbrace \frac{1}{s-\rho} \right\rbrace =  \frac{\sigma-\frac{1}{2}}{(\sigma-\frac{1}{2})^2+(t-\gamma)^2}.$$
To begin,
    \begin{align*}
        |Z_3| &= \left| \frac{1}{\log x}\sum_{\rho}\frac{x^{\rho-s}-x^{2(\rho-s)}}{(s-\rho)^2} \right| \\
        &\leq \frac{x^{\frac{1}{2}-\sigma}+x^{1-2\sigma}}{\log x} \sum_{\gamma}\frac{1}{(\sigma-\frac{1}{2})^2+(t-\gamma)^2}.
    \end{align*}
Since (\ref{zetas}) implies that for $|t|\geq 10$
    \begin{align*}
        \sum_{\rho} \text{Re}\left\lbrace \frac{1}{s-\rho} \right\rbrace  \leq \text{Re}\left\lbrace \frac{\zeta'}{\zeta}(s) \right\rbrace  + \frac{1}{2} \log |t| + 3,
    \end{align*}
and $\text{Re}(s)\leq |s|$ for all $s$, we have
    \begin{align*}
        |Z_3| &\leq \frac{x^{\frac{1}{2}-\sigma}+x^{1-2\sigma}}{(\sigma-\frac{1}{2})\log x} \left( \left| \frac{\zeta'}{\zeta}(s) \right|  + \frac{1}{2} \log |t| + 3 \right).
    \end{align*}

Substituting these bounds into (\ref{F(s,x)}), we have
\begin{align*}
    F(s,x) \leq \frac{x^{2-2\sigma} + x^{1-\sigma}}{t^2 \log x} + \frac{c_1}{(t^2+4) x^{2+\sigma} \log x} + \frac{x^{\frac{1}{2}-\sigma}+x^{1-2\sigma}}{(\sigma-\frac{1}{2})\log x} \left( \left| \frac{\zeta'}{\zeta}(s) \right|  + \frac{1}{2} \log |t| + 3 \right)
\end{align*}
for $x\geq x_0$, $\frac{1}{2}< \sigma \leq 1$, and $|t|\geq 10$. If we further impose $\sigma\geq \frac{1}{2} + \frac{\alpha}{\log x}$ with some $\alpha>0$, the bound simplifies to
\begin{align*}
    F(s,x) \leq \frac{e^{-2\alpha} x + e^{-\alpha}\sqrt{x}}{t^2 \log x} + \frac{c_1 e^{-\alpha}}{(t^2+4) x^{5/2}\log x} + \frac{e^{-\alpha}+e^{-2\alpha}}{(\sigma-\frac{1}{2})\log x} \left( \left| \frac{\zeta'}{\zeta}(s) \right|  + \frac{1}{2} \log |t| + 3 \right).
\end{align*}
Let $A_1=e^{-\alpha}+e^{-2\alpha}$ and $A_2 = \frac{\alpha}{\alpha - A_1}$. Assuming $\alpha > A_1$, which is satisfied for $\alpha\geq 0.722$, the above can be rearranged and simplified to
\begin{align*}
    \left| \frac{\zeta'}{\zeta}(s) \right| \leq A_2 \left| \sum_{n<x^2} \frac{\Lambda_x(n)}{n^s} \right| + \frac{A_1 A_2 \left( \log |t| + 6 \right)}{2(\sigma-\frac{1}{2})\log x}  + \frac{A_1 A_2 x}{t^2 \log x} + \frac{A_2c_1 e^{-\alpha}}{t^2 x^\frac{5}{2}\log x}.
\end{align*}
\end{proof}

\begin{Lem}\label{Lem4}
    Assume RH. For $T\geq T_0$, $\frac{1}{2}+\frac{\alpha}{\log T} \leq \sigma \leq \frac{3}{4}$, and $\alpha\geq 0.722$,
    \begin{equation}
        \int_0^T \left| \frac{\zeta'}{\zeta}(\sigma+it) \right|^2 dt \leq \frac{A_4 T}{(\sigma-1/2)^2},
    \end{equation}
    where $A_4$ is dependent on $\alpha$, and given in (\ref{A5}). For $T_0=10^3$ we can take $A_4 = 0.576$ with $\alpha=37$, or for $T_0=10^8$ we can take $A_4 = 0.535$ with $\alpha=26$. 
\end{Lem}
\begin{proof}
    Using the Cauchy--Schwarz inequality, Lemma \ref{Lem3} implies
    \begin{align}\label{log-zeta squared}
        \left| \frac{\zeta'}{\zeta}(\sigma+it) \right|^2 &\leq 2A_2^2 \left| \sum_{n<x^2} \frac{\Lambda_x(n)}{n^s} \right|^2  +  \frac{2A_1^2 A_2^2 \left( \log |t| + 6 \right)^2 }{(\sigma-\frac{1}{2})^2\log^2 x}  + \frac{8A_1^2 A_2^2 x^2}{t^4 \log^2 x} + \frac{8 A_2^2 c_1^2  e^{-2\alpha}}{t^4 x^5\log^2 x},
    \end{align}
    over the same range of variables and with the same constants as defined in Lemma \ref{Lem3}. As (\ref{log-zeta squared}) holds over $|t|\geq 10$, it can be integrated over $t\in[10, T]$, for some $T>10$. For the first term,
    \begin{align}\label{lambda-mean-value}
         \int_{10}^T \left| \sum_{n<x^2} \frac{\Lambda_x(n)}{n^s} \right|^2 dt &= (T-10) \sum_{n<x^2} \frac{\Lambda_x(n)^2}{n^{2\sigma}} + \sum_{\substack{m,n<x^2 \\ m\neq n}} \frac{\Lambda_x(m)\Lambda_x(n)}{(mn)^{\sigma}}  \int_{10}^T \left(\frac{n}{m}\right)^{it} dt \\ \nonumber
         &< T \sum_{n<x^2} \frac{\Lambda_x(n)^2}{n^{2\sigma}} + 2\sum_{\substack{m,n<x^2 \\ m\neq n}} \frac{\Lambda_x(m)\Lambda_x(n)}{(mn)^{\sigma} \left| \log (m/n) \right|}.
    \end{align}
The first sum in (\ref{lambda-mean-value}) can be bounded with 
\begin{align*}
    \sum_{n<x^2} \frac{\Lambda_x(n)^2}{n^{2\sigma}} &< \sum_{n=1}^\infty \frac{\Lambda(n)\log n}{n^{2\sigma}} = \frac{d}{ds} \left[ \frac{\zeta'}{\zeta}(s) \right]_{s=2\sigma}.
\end{align*}
For non-trivial zeros $\rho$, and all $s\in\mathbb{C}$, it is known that (e.g. see (8) and (9) in \cite[p.~80]{Davenport_1980}) 
\begin{equation}
    \frac{d}{ds}\left(\frac{\zeta'(s)}{\zeta(s)} \right) = \frac{1}{(s-1)^2} - \sum_{n=1}^\infty \frac{1}{(s+2n)^2} - \sum_\rho \frac{1}{(s-\rho)^2}.
\end{equation}
Thus, for $\sigma>\frac{1}{2}$,
\begin{align*}
    \sum_{n<x^2} \frac{\Lambda_x(n)^2}{n^{2\sigma}} &< \frac{1}{(2\sigma-1)^2} + \sum_{n=1}^\infty \frac{1}{(2\sigma+2n)^2} + \sum_\gamma \frac{1}{|2\sigma-1/2+i\gamma|^2} \\
    &\leq \frac{1}{4(\sigma-1/2)^2} + \frac{\pi^2}{8} - 1 + \sum_\gamma \frac{1}{\gamma^2},
\end{align*}
where $\sum_\gamma 1/\gamma^2 < c_0 = 0.04621$, computed in \cite[Cor.~1]{B_P_T_21}.

For the second sum in (\ref{lambda-mean-value}) we can use $\log \lambda> 1 - \lambda^{-1}$ over $\lambda>1$, so for $\sigma>1/2$,
\begin{align}\label{lambdas}
    \sum_{\substack{m,n<x^2 \\ m\neq n}} \frac{\Lambda_x(m)\Lambda_x(n)}{(mn)^{\sigma} \left| \log (m/n) \right|} < \log^2x \sum_{\substack{m,n<x^2 \\ m\neq n}} \left( \frac{1}{\sqrt{mn}} + \frac{1}{|m-n|} \right).
\end{align}
Note that the bound on $|\log(m/n)|$ holds for both $m>n$ and $n>m$ because of the symmetry in the resulting expression. For $x\geq 1$, partial summation gives 
\begin{equation*}
    \sum_{n<x^2} \frac{1}{\sqrt{n}} < 2x,
\end{equation*}
and the Euler–Maclaurin formula gives
\begin{align*}
    \sum_{\substack{m,n<x^2 \\ m\neq n}} \frac{1}{|m-n|} = 2\sum_{m<x^2} \sum_{n<m} \frac{1}{m-n} &< 2\sum_{m<x^2} \sum_{k<x^2} \frac{1}{k} < 4x^2\log x + 2\gamma x^2 + 1,
\end{align*}
whence we have
\begin{align*}
    \sum_{\substack{m,n<x^2 \\ m\neq n}} \frac{\Lambda_x(m)\Lambda_x(n)}{(mn)^{\sigma} \left| \log (m/n) \right|} &<  4x^2\log^3 x + (4+2\gamma)x^2\log^2x + \log^2x.
\end{align*}
Returning to (\ref{lambda-mean-value}), we have
    \begin{align*}
         \int_{10}^T \left| \sum_{n<x^2} \frac{\Lambda_x(n)}{n^s} \right|^2 dt &< T \left( \frac{1}{4(\sigma-1/2)^2} + A_3 \right) + 8c_2 x^2\log^3 x,
    \end{align*}
where $A_3 = \pi^2/8 - 1 + c_0$ and $c_2 = 2.25$ for $x\geq 3$. Using this in (\ref{log-zeta squared}) gives
\begin{align*}
    \frac{1}{A_2^2} \int_{10}^T \left| \frac{\zeta'}{\zeta}(\sigma+it) \right|^2 dt &\leq \frac{T}{2(\sigma-\frac{1}{2})^2} + 2 A_3 T + 16 c_2 x^2\log^3x  \\
    &\quad + \frac{2A_1^2}{(\sigma-\frac{1}{2})^2\log^2 x} \int_{10}^T \left( \log t + 6 \right)^2 dt + 8 \left( \frac{A_1^2 x^2}{\log^2 x} +  \frac{c_1^2 e^{-2\alpha}}{x^5\log^2 x}  \right) \int_{10}^T \frac{1}{t^4} dt  \\
    &\leq \frac{T}{2(\sigma-\frac{1}{2})^2} + 2 A_3 T + 16 c_2 x^2\log^3x  +  \frac{2A_1^2 c_3 T\log^2T}{(\sigma - \frac{1}{2})^2 \log^2 x}  \\
    &\quad +  \frac{1}{375} \left( \frac{A_1^2 x^2}{\log^2 x} +  \frac{c_1^2 e^{-2\alpha}}{x^5\log^2 x}  \right),
\end{align*}
where we can take $c_3 = 1 + \frac{10}{\log T_0} + \frac{26}{\log^2 T_0}$ for any $T_0 > 10$. Also note that the $1/375$ comes from estimating the second integral over $t$. We will now take $x=T^{\nu}$, and for $x\geq x_0$ and $T\geq T_0$, this will be for any $\nu\geq \log x_0/\log T_0$. The previous bound becomes
\begin{align*}
    \frac{1}{A_2^2} \int_{10}^T \left| \frac{\zeta'}{\zeta}(\sigma+it) \right|^2 dt &\leq \left( \frac{1}{2} + \frac{2A_1^2 c_3}{\nu^2} \right) \frac{T}{(\sigma-\frac{1}{2})^2}  +  2 A_3 T + 16c_2 \nu^3 T^{2\nu} \log^3T  \\
    &\quad  +  \frac{A_1^2 T^{2\nu}}{375\nu^2 \log^2 T}  +  \frac{c_1^2 T^{1-5\nu}}{375e^{2\alpha}\nu^2 \log^2 T}.
\end{align*}

It remains to estimate the integral over $t\in[0,10]$. By the maximum modulus principle, $$\left| \frac{\zeta'}{\zeta}(s) \right| \leq \text{max}_{z\in \delta S} \left| \frac{\zeta'}{\zeta}(z) \right|,$$ where $\delta S$ is the boundary of $S:=\{ z\in\mathbb{C}: \frac{1}{2}< \sigma\leq \frac{3}{4}, 0\leq t\leq 10 \}$. This implies
\begin{align*}
    \int_{0}^{10} \left| \frac{\zeta'}{\zeta}(\sigma+it) \right|^2 dt &\leq 10 \cdot 4.7^2 < 215.
\end{align*}
Therefore, for all $T\geq T_0$ and $1/2<\sigma\leq 3/4$ we have
\begin{align*}
    \int_0^T \left| \frac{\zeta'}{\zeta}(\sigma+it) \right|^2 dt \leq  \frac{A_4 T}{(\sigma-\frac{1}{2})^2},
\end{align*}
where
\begin{align}\label{A5}
    A_4 =  A_2^2\left( \frac{1}{2} + \frac{2A_1^2 c_3}{\nu^2}  +   \frac{A_3}{8}  +  \left( c_2 \nu^3 \log^3T_0   +  \frac{A_1^2/6000}{\nu^2 \log^2 T_0} \right) \frac{1}{T_0^{1-2\nu}}  +   \frac{c_1^2 e^{-2\alpha} /6000 }{\nu^2 T_0^{5\nu}\log^2 T_0}  \right) + \frac{215}{16T_0} ,
\end{align}
for any $\nu \in \left[ \log x_0/\log T_0, 1/2\right)$ and $T_0\geq \exp(\frac{3}{1-2\nu})$. The latter condition is needed to ensure the $T^{2\nu-1}\log^3T$ term is decreasing for all $T\geq T_0$. Optimising $\nu$ and $\alpha$ with $x_0=3$ and $T_0=10^3$, we can take $A_4 = 0.576$ with $\nu = 0.1591$ and $\alpha=37$. This constant $A_4$ approaches its limit relatively quickly as $T_0$ increases, so with $T_0=10^8$ we can take $A_4 = 0.535$ with $\nu = 0.0597$ and $\alpha=26$. \end{proof}

For any choice of $T_0\geq 10^8$, $A_4$ is within $10^{-6}$ of its limit. Larger $x_0$ also does not reduce $A_4$. In fact, we see the opposite. A smaller $x_0$ allows smaller admissible $\nu$, which reduces the terms with a factor of $T^{\nu}$. As a result, the optimal value of $\nu$ in both cases is its lower limit. One of the most direct ways to reduce $A_4$ would be the use of a smaller upper bound on $\sigma$.

\subsection{Proof of Theorem \ref{mean-value-thm}}

To begin, let $$\theta_0(x) = \frac{1}{2} \lim_{\varepsilon\rightarrow 0}\left( \theta(x+\varepsilon)+\theta(x-\varepsilon) \right),$$ so $\theta_0(x) = \theta(x)$ except when $x$ is prime, and let $G(y,\delta) = \theta(y+\delta y)- \theta(y) - \delta y$ for any $\delta\in (0,1]$. By Perron's formula, we can write, for $x>1$, $s=\sigma+it$, and prime $p$,
\begin{align}\label{perron}
    \theta_0(x) = \frac{1}{2\pi i} \int_{2-i\infty}^{2+i\infty} \frac{x^s}{s} \bigg\{ \sum_p \frac{\log p}{p^s} \bigg\} ds.
\end{align}
As the integral is over $\sigma \geq 2$, the sum can be re-written as $$\sum_p \frac{\log p}{p^s} = \sum_{n=2}^\infty \frac{\Lambda(n)}{n^s} - \sum_{r=2}^\infty \sum_p \frac{\log p}{p^{rs}} = -\frac{\zeta'(s)}{\zeta(s)} - g(s).$$ As $g(s)$ is convergent for $\sigma>1/2$, it can be bounded over this region with
\begin{align*}
|g(s)| = \left| \sum_p \frac{\log p}{p^s(p^s-1)} \right| \leq \sum_p \frac{\log p}{p^\sigma(p^\sigma-1)} &\leq c_4 \sum_{p\geq 19} \frac{\log p}{p^{2\sigma}} +  \sum_{2\leq p< 19} \frac{\log p}{p^\sigma(p^\sigma-1)}  \\
&\leq c_4\sum_{n=2}^\infty \frac{\Lambda(n)}{n^{2\sigma}} + \sum_{2\leq p<19} \left( \frac{\log p}{p^\sigma(p^\sigma-1)} - \frac{c_4 \log p}{p^{2\sigma}} \right)
\end{align*}
where $c_4= \sqrt{19}/(\sqrt{19}-1) \approx 1.2978$. This simplifies to
\begin{align*}
    |g(s)| < - c_4\frac{\zeta'(2\sigma)}{\zeta(2\sigma)} + 1.4255.
\end{align*}
By Delange's theorem \cite{Delange_87}, and the second display equation on p. 334 of \cite{Delange_87}, we have 
\begin{equation*}
    - \frac{\zeta'(\sigma)}{\zeta(\sigma)} < \frac{1}{\sigma - 1}
\end{equation*}
for all $\sigma>1$. Hence, for $\sigma\in (1/2, 3/4]$ and $c_5 = 1.0053$ we can use
\begin{equation} \label{g}
    |g(s)| < \frac{c_4}{2\sigma-1} + 1.4255 < \frac{c_5}{\sigma - \frac{1}{2}}.
\end{equation}

Returning to (\ref{perron}), we will move the line of integration to some $\sigma\in (1/2,3/4]$. Part of this process involves evaluating a closed contour integral of the integrand in (\ref{perron}) over at most $1/2 < \text{Re}(s) \leq 2$ and all $t$. The only pole of the integrand in this region is at $s=1$. Therefore, by Cauchy's residue theorem,
\begin{align}\label{contoured}
    \theta_0(x) - x = - \frac{1}{2\pi} \int_{-\infty}^{\infty} \frac{x^{s}}{s} \bigg\{ \frac{\zeta'(s)}{\zeta(s)} + g(s) \bigg\} dt.
\end{align}
We will use (\ref{contoured}) to set up an expression for the error term of $\theta(x)$ in intervals. Let $\kappa$ be defined such that $e^\kappa = 1+\delta$, meaning $0< \kappa \leq \log 2$. For $\tau>0$, (\ref{contoured}) implies
\begin{equation*}
    \frac{\theta_0(e^{\kappa+\tau})-\theta_0(e^\tau)-\delta e^\tau}{e^{\sigma\tau}} = - \frac{1}{2\pi} \int_{-\infty}^{\infty} \frac{e^{\kappa s}-1}{s} e^{it\tau} \left( \frac{\zeta'}{\zeta}(s) + g(s) \right) dt.
\end{equation*}
By Plancherel's theorem\footnote{Selberg actually states that this step is justified by Parseval's theorem, but other sources refer to this theorem as Plancherel's.} (see e.g. \cite[Thm.~2, p.69]{Wiener_88})
\begin{align}\label{Plant}
    \int_{-\infty}^\infty \left| \frac{\theta_0(e^{\kappa+\tau})-\theta_0(e^\tau)-\delta e^\tau}{e^{\sigma\tau}} \right|^2 d\tau = \frac{1}{2\pi} \int_{-\infty}^\infty \left| \frac{e^{\kappa s}-1}{s} \right|^2 \left| \frac{\zeta'}{\zeta}(s) + g(s) \right|^2  dt.
\end{align}
Since $\theta_0(x) = \theta(x)$ almost everywhere, this statement is equally true for $\theta(x)$. As $|z|^2= z \overline{z}$ for any complex $z$, the integrals in (\ref{Plant}) are symmetric around 0, which implies
\begin{align}\label{post-plant}
    \int_0^\infty \left| \frac{G(e^{\tau}, \delta)}{e^{\sigma\tau}} \right|^2 d\tau &< \frac{1}{\pi} \int_{0}^\infty \left| \frac{e^{\kappa s}-1}{s} \right|^2 \left( \left| \frac{\zeta'(s)}{\zeta(s)} \right|^2 + \left| g(s) \right|^2 \right)  dt \nonumber \\ 
    &= \frac{1}{\pi} \sum_{k=0}^\infty \int_{(2^k-1)/\delta}^{(2^{k+1}-1)/\delta} \left| \frac{e^{\kappa s}-1}{s} \right|^2 \left( \left| \frac{\zeta'(s)}{\zeta(s)} \right|^2 + \left| g(s) \right|^2 \right)  dt.
\end{align}
The first factor in the integrand can be bounded in two different ways. The first uses the Taylor series for $e^x$, and is valid for all $\sigma\leq 3/4$,
\begin{align}\label{kap_bound}
    \left| \frac{e^{\kappa s}-1}{s} \right| = \left| \frac{1}{s} \sum_{n=1}^\infty \frac{(\kappa s)^n}{n!} \right| = \left| \sum_{n=0}^\infty \frac{\kappa^{n+1}s^n}{(n+1)!} \right| \leq \kappa \left| \sum_{n=0}^\infty \frac{(\kappa s)^n}{n!} \right| \leq e^{\frac{3}{4}\kappa}\kappa.
\end{align}
The second is more direct, and valid for all $\sigma\in(1/2,3/4]$,
\begin{align} \label{t_bound}
    \left| \frac{e^{\kappa s}-1}{s} \right| \leq  \frac{e^{\kappa \sigma} + 1}{|s|} \leq \frac{e^{\frac{3}{4}\kappa} + 1}{t}.
\end{align}
We will use (\ref{kap_bound}) for the $k=0$ term in (\ref{post-plant}), and (\ref{t_bound}) for the other terms, so as to have a convergent sum. Incorporating Lemma \ref{Lem4}, we have
\begin{align} \label{final-unsimplified}
    \int_0^\infty \left| \frac{G(e^{\tau}, \delta)}{e^{\sigma\tau}} \right|^2 d\tau &< \frac{(A_4+c_5^2)e^{\frac{3}{2}\kappa}\kappa^2}{\pi\delta(\sigma-\frac{1}{2})^2}  +  \frac{(e^{\frac{3}{4}\kappa}+1)^2\delta}{\pi (\sigma-\frac{1}{2})^2} \sum_{k=1}^\infty \frac{A_4(2^{k+1}-1) + c_5^2\cdot 2^k}{(2^k-1)^2} \\ \nonumber
    &\leq \left( (1+10^{-3}) (A_4+c_5^2) + (4+10^{-2})(A_4 A_5 + c_5^2 A_6) \right) \frac{\delta}{\pi (\sigma-\frac{1}{2})^2}
\end{align}
for $\delta\leq T_0^{-1}\leq 10^{-3}$, where $A_5 = 4.35..$ and $A_6 = 2.74..$ are the two convergent sums in (\ref{final-unsimplified}). In using Lemma \ref{Lem4} we assumed $(2^{k+1}-1)/\delta \geq T_0$, which is true over $k\geq 0$ with $\delta \leq T_0^{-1}$.

We can now choose $\sigma$ to minimise the final bound. By the upper bound on $\sigma$ in Lemma \ref{Lem4}, we can take $\sigma = \frac{1}{2} + \frac{\alpha}{\log(1/\delta)}$ for $\delta \leq \text{min}(e^{-4\alpha},T_0^{-1})$ and $\alpha$ defined as in Lemma \ref{Lem4}. Also, to simplify the integral of interest, let $y=e^\tau$, so for $y>1$ we have
\begin{align*}
    \int_0^\infty \left| \frac{G(e^{\tau}, \delta)}{e^{\sigma\tau}} \right|^2 d\tau = \int_1^\infty \left| \frac{G(y, \delta)}{y^{\sigma+\frac{1}{2}}} \right|^2 dy = \int_1^\infty \left| \frac{G(y, \delta)}{y^{1+\frac{\alpha}{\log(1/\delta)}}} \right|^2 dy.
\end{align*}
We can now use the bound in (\ref{final-unsimplified}) for a version of the above integral over a finite range of $y$. For applications, it is useful if the range of integration is a function of $\delta$. The importance of $\delta$ here is because $G(y,\delta)$ is the error in the PNT over an interval defined by $\delta$. Let $b$ be a positive parameter to write
\begin{align*}
    \int_1^\infty \left| \frac{G(y, \delta)}{y^{1+\frac{\alpha}{\log(1/\delta)}}} \right|^2 dy &> \int_1^{\delta^{-b}} y^{\frac{2\alpha}{\log \delta}} \left| \frac{G(y, \delta)}{y} \right|^2 dy \\
    &> \delta^{\frac{-2\alpha b}{\log \delta}} \int_1^{\delta^{-b}} \left| \frac{G(y, \delta)}{y} \right|^2 dy = \frac{1}{e^{2\alpha b}} \int_1^{\delta^{-b}} \left| \frac{G(y, \delta)}{y} \right|^2 dy.
\end{align*}
Hence by (\ref{final-unsimplified}) we can conclude
\begin{align}\label{delta_ver}
    \int_1^{\delta^{-b}} \left| \frac{G(y, \delta)}{y} \right|^2 dy < e^{2\alpha(b-1)} A_7 \delta\log^2(1/\delta)
\end{align}
for $$A_7 = \frac{e^{2\alpha}}{\alpha^2 \pi} \left( (1+10^{-3}) (A_4+c_5^2) + (4+10^{-2})(A_4 A_5 + c_5^2 A_6) \right).$$ This is an explicit form of (13) in \cite{Selberg_43}. This result is more precise than Theorem \ref{mean-value-thm} for fixed $\delta$, but it is not as simple to use. Theorem \ref{mean-value-thm} comes from a slightly different choice of $\sigma$: we instead take $\sigma = \frac{1}{2} + \frac{\alpha}{\log x}$ in (\ref{final-unsimplified}) for $x\geq \max\{ e^{4\alpha}, T_0\}$ (by the bounds on $\sigma$ from Lemma \ref{Lem4}). By the same steps as above, we reach 
\begin{align*}
    \int_0^\infty \left| \frac{G(e^{\tau}, \delta)}{e^{\sigma\tau}} \right|^2 d\tau = \int_1^\infty \left| \frac{G(y, \delta)}{y^{1+\frac{\alpha}{\log x}}} \right|^2 dy > \frac{1}{e^{2\alpha}x^2} \int_1^{x} \left| G(y, \delta) \right|^2 dy,
\end{align*}
and hence
\begin{align}\label{final-mean-value}
    \int_1^{x} \left| G(y, \delta) \right|^2 dy < A_7 \delta x^2 \log^2 x
\end{align}
for $\delta \leq T_0^{-1}$. Using $T_0=10^8$ in (\ref{A5}), and optimising over $\alpha$ and $\nu$, we can take $A_7 = 202$ with $\alpha=2.08$ and $\nu = 0.285$, such that (\ref{final-mean-value}) holds for all $x\geq 10^8$.

\section{Primes in some short intervals}\label{proof-corollary}

Theorem \ref{mean-value-thm} can be used to find an explicit estimate for the exceptional set of primes in short intervals of length $h$ where $h/\log y \rightarrow \infty$. This is demonstrated in Corollary \ref{some-intervals}.
\begin{manualCor}{1}
    Assuming RH, the set of $y\in [x,2x]$ for which there is at least one prime in $(y, y+ 32277 \log^2 y]$ has a measure of at least $x/2$ for all $x\geq 2$.
\end{manualCor}

\begin{proof}
As before, let $G(y,\delta) = \theta(y+\delta y)- \theta(y) - \delta y$ for $\delta\in (0,1]$. Also let $\delta_1 = \frac{\lambda\log^2(2x)}{2x}$ for $x>1$ and $\lambda\geq 1$. We will use the alternative version of Theorem \ref{mean-value-thm} from the previous section, stated in (\ref{delta_ver}), which is under RH. Taking $\delta = \delta_1$ in (\ref{delta_ver}) gives
\begin{align}\label{delta_ver_interval}
\int_1^{\delta_1^{-b}} \left| \frac{G(y, \delta)}{y} \right|^2 dy < e^{2\alpha(b-1)} A_7 \frac{\lambda\log^2(2x)}{2x} \log^2\left(\frac{2x}{\lambda\log^2(2x)}\right)
\end{align}
for any $b>0$ and all $x$ for which $\delta_1\leq \min\{e^{-4\alpha},T_0^{-1}\}$. The constants $T_0$ and $\alpha$ are determined by Lemma \ref{Lem4}, and correspond to the $A_4$ in the definition of $A_7$. Choosing $b>1$ will make $\delta_1 < (2x)^{-\frac{1}{b}}$ for sufficiently large $x$, allows us to write
\begin{equation*}
    \int_1^{\delta_1^{-b}} \left| \frac{G(y, \delta)}{y} \right|^2 dy > \int_{x}^{2x} \left| \frac{G(y, \delta)}{y} \right|^2 dy.
\end{equation*}
Imposing this restriction on $b$ thus implies
\begin{align*}
    \int_{x}^{2x} \left| G(y, \delta) \right|^2 dy &< 2 e^{2\alpha(b-1)} A_7 \lambda x \log^2(2x) \log^2x  <  2 e^{2\alpha(b-1)} A_7 \lambda x \log^4x
\end{align*}
over $\delta_1<\text{min}\{e^{-4\alpha}, T_0^{-1}, (2x)^{-\frac{1}{b}} \}$ and $x\geq 3$.

We can use this bound to prove that for a subset of $y\in[x, 2x]$ of measure $\geq (1-g)x$, with $g\in(0,1)$ and sufficiently large $x$, we have
\begin{equation}\label{almost-all}
    \left| G(y, \delta) \right|^2 < B \log^4 y
\end{equation}
for some $B>0$. To justify this, suppose for a contradiction that there exists a subset $I$ of $y\in[x, 2x]$ of measure $\geq gx$ for which $$\left| G(y, \delta) \right|^2 \geq B \log^4 y.$$ This would imply
\begin{align*}
    \int_{x}^{2x} \left| G(y, \delta) \right|^2 dy &\geq  B \log^4x \int_I dy = Bg x \log^4x.
\end{align*}
This will be a contradiction for $B\geq 2 e^{2\alpha(b-1)} A_7 \lambda/g$. Therefore, choosing the smallest possible $B$, we have for $x \leq y\leq 2x$,
\begin{align*}
    \theta(y+ \lambda\log^2 y) - \theta(y)  &\geq  \theta\left( y +  \frac{\lambda\log^2(2x)}{2x}y \right) - \theta(y) - \frac{\lambda\log^2(2x)}{2x}y + \frac{\lambda\log^2(2x)}{2x}y \\
    &> \left(- \sqrt{B} + \frac{\lambda}{2} \right)\log^2 y,
\end{align*}
which implies that there will be at least one prime in the interval $(y, y+\lambda\log^2 y]$ for $$-\sqrt{\frac{2 e^{2\alpha(b-1)} A_7 \lambda}{g}} + \frac{\lambda}{2} > 0.$$ We can conclude that, under the assumption of RH, the set of $y \in [x,2x]$ for which there are primes in $(y, y+ \lambda\log^2 y]$ has measure $\geq (1-g)x$ for $$\lambda > \frac{8 e^{2\alpha(b-1)} A_7}{g}.$$ To prove Corollary \ref{some-intervals}, we take $g=1/2$ and optimise the lower bound on $\lambda$ over $b$. We aim to find the smallest $\lambda$ for which the condition on $\delta_1$ holds over $x\geq 4\cdot 10^{18}$, as the computations of Oliveira e Silva, Herzog, and Pardi \cite{O_H_P_14} can be used to verify Corollary \ref{some-intervals} for $x<4\cdot 10^{18}$. We can also re-optimise $\alpha$, and use a higher $T_0$ in Lemma 4 than used to reach Theorem \ref{mean-value-thm}, as it just needs to satisfy the condition on $\delta_1$.

For $T_0 = 1.3\cdot 10^{11}$ we find that we can take $A_7 = 236.72$ with $\alpha = 1.5295$, and with $b=1.700423$ we have $\lambda = 32277$. This was achieved using a partially manual optimisation process, in that the lower bound on $\lambda$ was first minimised over $\alpha$ using an in-built optimising function and some guess for $b$. The guess for $b$ was then adjusted until it satisfied $\delta_1<\text{min}\{e^{-4\alpha}, T_0^{-1}, (2x)^{-\frac{1}{b}} \}$, which required a guess of an upper bound for $\lambda$. After a valid solution set was found, it was refined computationally in Python.

The computations in Section 2.2 of \cite{O_H_P_14} confirm that the interval $(y, y+ 32277\log^2 y]$ contains a prime for all $2\leq y\leq 4\cdot 10^{18}$. More specifically, the calculations in Section 2.2.1 of \cite{O_H_P_14} show that there is a prime in $(y,y+2.09\log^2y]$ for all $2\leq y\leq 4\cdot 10^{18}$.
\end{proof}

The trade-off in this type of result is between the length of the interval, the size of the exceptional set, and the range for which the result holds. Corollary \ref{some-intervals} was built first on the asymptotic length of the interval, then the desired measure of the exceptional set, and lastly the constant in the interval, which was calculated based on the smallest $x$ for which we wanted the result to hold. An alternative would have been to first fix the constant in the interval, then calculate the measure of the exceptional set. It would also be possible to consider an asymptotically larger interval than $O(\log^2 y)$. In this case, the exceptional set would be asymptotically smaller than $x$, and could be given explicitly using Theorem \ref{mean-value-thm} and the working in \cite[p.~11]{Selberg_43}.

\section{An explicit bound for Goldbach numbers}\label{proof-goldbach-intervals}

A Goldbach number is an even positive integer that can be written as the sum of two odd primes. With Theorem \ref{mean-value-thm} we can prove Theorem \ref{thm-Goldbach-intervals}, restated here.
\begin{manualThm}{2}
Assuming RH, there exists a Goldbach number in the interval $(x,x+ 9696 \log^2 x]$ for all $x\geq 2$.
\end{manualThm}

\noindent By Theorem \ref{mean-value-thm} and RH, we can state that for $x \geq 10^8$, $\delta\in (0,10^{-8}]$, and any $a\in [10^{-8}, 1)$, 
\begin{equation}\label{Goldbach-mean-value}
\int_{ax}^{x} | \theta(t+\delta t) - \theta(t) - \delta t|^2 dt \leq 202 \delta x^2 \log^2 x.    
\end{equation}
To prove Theorem \ref{thm-Goldbach-intervals} we largely follow the proof of Montgomery and Vaughan's Theorem 2 in \cite[Sect.~9]{M_V_75}, and use the above bound. We also optimise a few choices in the proof.

Suppose the interval $(x,x+h]$ contains no sum of two primes for $1\leq h\leq x$. Then, for any $y$, at least one of the two intervals 
\begin{equation}\label{pairs}
    \left(y, y+\frac{1}{2} h \right], \left(x-y, x-y+\frac{1}{2}h \right]
\end{equation}
will not contain a prime number. Both of these intervals can be represented by 
\begin{equation*}
\left(\frac{x}{2}+\frac{kh}{2}, \frac{x}{2} + \frac{(k+1) h}{2} \right],
\end{equation*}
and for any choice of $k\in K = \left[(2a-1) xh^{-1}  + 1, (1-2a)xh^{-1} - 1\right]$ with $a\in(0,1/2]$, which defines one of the intervals in (\ref{pairs}), there exists another $k\in K$ which defines the other interval such that both intervals lie in $(ax,x-ax]$. Each of these pairs of intervals lies symmetrically around the midpoint of $(ax,x-ax]$, so it is possible to completely cover $(ax,x-ax]$ with at most $(1-2a)x/h$ pairs of the form (\ref{pairs}). Therefore, at least $(1-2a)x/h$ of these intervals covering $(ax,x-ax]$ do not contain a prime.

As these intervals span $h/2$, we can write that for $\delta = \delta(x)\leq h/2x$ $$| \theta(t+\delta t) - \theta(t) - \delta t| = \delta t$$ on a set $I$ of $t\in (ax,x-ax]$ of measure $(\frac{1}{2} - a)x$. Note that the condition on $\delta$ is a result of requiring $\delta t\leq h/2$ for all $t$. Therefore, we have
\begin{align*}
\int_{ax}^{x} | \theta(t+\delta t) - \theta(t) - \delta t|^2 dt &> \int_I | \theta(t+\delta t) - \theta(t) - \delta t|^2 dt = \delta^2 \int_I t^2 dt \\
&> \delta^2\int_{ax}^{x/2} t^2 dt = \frac{\delta^2 x^3}{3} \left( \frac{1}{8} - a^3 \right).
\end{align*}
Taking $\delta = h/2x$ with $h=C\log^2 x$, this bound contradicts (\ref{Goldbach-mean-value}) for $$C\geq \frac{6\cdot 202}{\frac{1}{8}-a^3}$$ and all $x$ satisfying $$\frac{\log^2 x}{x} \leq \frac{2}{10^{8}C}.$$ The lower bound on $C$ is minimised at the smallest $a=10^{-8}$, meaning we can take $C=9696$. Therefore, there must exist at least one Goldbach number in $(x,x+ 9696\log^2 x]$ for all $x\geq 6\cdot 10^{14}$. The computation in \cite{O_H_P_14} confirms the Goldbach conjecture up to $4\cdot 10^{18}$, so this interval must also contain a Goldbach number for all $ x\geq 2$.

\section*{Acknowledgements}
Many thanks to Tim Trudgian, Aleks Simoni\v{c}, and Daniel Johnston for their advice and suggestions. Thanks also to Nicol Leong for fixing the code.


\bibliographystyle{IEEEtranSN}
\bibliography{references_LR}

\end{document}